\documentclass[11pt]{article}
\usepackage{enumerate}
\usepackage{amssymb,a4wide,latexsym,makeidx,epsfig}
\usepackage{amsthm}
\usepackage{amsmath}
\usepackage{lipsum} 
\usepackage{enumerate}
\usepackage{mathrsfs}
\usepackage{xcolor}
\usepackage{tikz}
\usepackage{geometry}
\usepackage{appendix}
\usepackage{microtype}
\usepackage[american]{babel}
\allowdisplaybreaks

\newtheorem{theorem}{Theorem}[section]

\newtheorem{definition}[theorem]{Definition}
\newtheorem{lemma}[theorem]{Lemma}

\newtheorem{proposition}[theorem]{Proposition}
\newtheorem{corollary}[theorem]{Corollary}

\begin{document}
\textwidth 150mm \textheight 225mm
\title{Integer colorings with no rainbow 3-term arithmetic progression
\thanks{Supported by the National Natural Science Foundation of China (No. 11871398)
and China Scholarship Council (No. 201906290174).}
}
\author{{Xihe Li$^{1,2}$, Hajo Broersma$^{2,}$\thanks{Corresponding author.}, Ligong Wang$^{1}$}\\
{\small $^{1}$ School of Mathematics and Statistics,}\\ {\small Northwestern Polytechnical University, Xi'an, Shaanxi 710129, PR China}\\
{\small $^{2}$ Faculty of Electrical Engineering, Mathematics and Computer Science,}\\ {\small University of Twente, P.O. Box 217, 7500 AE Enschede, The Netherlands}\\
{\small E-mail: lxhdhr@163.com; h.j.broersma@utwente.nl; lgwangmath@163.com}}
\date{}
\maketitle
\begin{center}
\begin{minipage}{120mm}
\vskip 0.3cm
\begin{center}
{\small {\bf Abstract}}
\end{center}
{\small
In this paper, we study the rainbow Erd\H{o}s-Rothschild problem with respect to 3-term arithmetic progressions. We obtain the asymptotic number of $r$-colorings of $[n]$ without rainbow 3-term arithmetic progressions, and we show that the typical colorings with this property are 2-colorings. We also prove that $[n]$ attains the maximum number of rainbow 3-term arithmetic progression-free $r$-colorings among all subsets of $[n]$. Moreover, the exact number of rainbow 3-term arithmetic progression-free $r$-colorings of $\mathbb{Z}_p$ is obtained, where $p$ is any prime and $\mathbb{Z}_p$ is the cyclic group of order $p$.
}
\end{minipage}
\end{center}

\section{Introduction}
\label{sec:ch-intro}

Two fundamental topics in extremal combinatorics are counting discrete structures that have certain properties and analyzing their typical properties. One of the major problems, dating back to the 1970s and initiated by Erd\H{o}s, Kleitman and Rothschild \cite{ErKR}, is to determine the number of $H$-free graphs on $n$ vertices and establish their typical structure. We refer to \cite{BLPS,BMSW,KOTZ,MoSa} and references therein for related results. In the setting of integers, this problem was introduced by Cameron and Erd\H{o}s \cite{CaEr} who studied the number of subsets of positive integers satisfying some constraint. We refer to \cite{BLST18,Gre,Sap} for the number of sum-free subsets, to~\cite{BaLS} for the number of subsets with no $k$-term arithmetic progression, and to~\cite{LiPa} for results on multiplicative Sidon sets.

In the context of colored discrete structures, Erd\H{o}s and Rothschild \cite{Erd74} asked which graphs on $n$ vertices admit the maximum number of $r$-edge-colorings without a monochromatic subgraph $H$. This  Erd\H{o}s-Rothschild problem was extended to edge-colorings of graphs with other forbidden color patterns. For example, the number of rainbow triangle-free $r$-edge-colorings of complete graphs (also known as Gallai colorings) and their typical structure were determined in \cite{BaLi,BaBH}. This problem can also be generalized to other discrete structures. Hoppen, Kohayakawa and Lefmann \cite{HoKL} studied the Erd\H{o}s-Rothschild extension of the celebrated Erd\H{o}s-Ko-Rado theorem. Liu, Sharifzadeh and Staden \cite{LiSS} and H\`{a}n and Jim\'{e}nez \cite{HaJi} determined the maximum number of monochromatic sum-free colorings of integers and of finite abelian groups, respectively.  Recently, motivated by these results, Cheng et al. \cite{CJLWZ} studied the number of rainbow sum-free colorings of integers and their typical structure.

In this paper, we will focus on analogous problems with respect to rainbow 3-term arithmetic progressions. Let $n$ be a positive integer and let $[n]=\{1, 2, \ldots, n\}$. Given a subset $A\subseteq [n]$, a \emph{$k$-term arithmetic progression} ($k$-AP) of $A$ is a sequence in $A$ of the form $a, a+d, a+2d, \ldots, a+(k-1)d$, where $d\in [n]$. For an integer $r\geq 1$ and a subset $A\subseteq [n]$, let $c\colon\, A \to [r]$ be an \emph{$r$-coloring} of $A$. Given an $r$-coloring $c$ of $A$, a $k$-AP of $A$ is called \emph{rainbow} if $c(a+id)\neq c(a+jd)$ for every $0\leq i<j\leq k-1$. An $r$-coloring of $A$ is called \emph{rainbow $k$-AP-free} if it contains no rainbow $k$-AP.

Already back in 1927, Van der Waerden \cite{vdW} proved that for any positive integers $k$ and $r$, if $n$ is sufficiently large, then every $r$-coloring of $[n]$ contains a monochromatic $k$-AP. A strengthening of this result was conjectured by Erd\H{o}s and Tur\'{a}n \cite{ErTu}, and proved by Szemer\'{e}di \cite{Sze} in 1975.

\noindent\begin{theorem}\label{th:Sze} {\normalfont (Szemer\'{e}di's theorem \cite{Sze})} Let $k$ be a positive integer and let $\delta >0$. Then there exists a positive integer $sz(k, \delta)$ such that for all $n\geq sz(k, \delta)$, every subset $C\subseteq [n]$ with $|C|> \delta n$ contains a $k$-AP.
\end{theorem}

In \cite{ErGr}, Erd\H{o}s and Graham proved a canonical version of Van der Waerden's theorem, that is, for any integer $k\geq 3$, if $n$ is sufficiently large, then every coloring (with any number of colors) of $[n]$ contains either a rainbow $k$-AP or a monochromatic $k$-AP. For more results related to rainbow arithmetic progressions, see \cite{AxFDF,BeSY,BEHH,JLMNR,LlMo,PaTo}.

For any $A\subseteq [n]$ and $r\geq 3$, we use $g_r(A, [n])$ (or simply $g_r(A)$) to denote the number of rainbow 3-AP-free $r$-colorings of $A$. By choosing two of the $r$ colors and coloring the elements of $A$ arbitrarily with these two colors, a lower bound on $g_r(A)$ is
\begin{equation}\label{eq:1}
g_r(A)\geq \binom{r}{2}(2^{|A|}-2)+r=\binom{r}{2}2^{|A|}-r^2+2r.
\end{equation}
For the upper bound, we will prove the following result in Section~\ref{sec:proof}.

\noindent\begin{theorem}\label{th:3-AP} For all integers $r\geq 3$ and any real number $\xi$ with $0<\xi \leq \frac{3}{13+8\log_2 r}$, there exists $n_0\in\mathbb{N}$ such that for all $n\geq n_0$ the following holds. If $A\subseteq [n]$ and $|A|\geq (1-\xi)n$, then
$$g_r(A)\le\binom{r}{2} 2^{|A|}+2^{-\frac{n}{36\log_2n}}2^n.$$
\end{theorem}

In view of inequality (\ref{eq:1}), our upper bound given by Theorem \ref{th:3-AP} is asymptotically tight for $A=[n]$, and the description of the typical structure follows immediately.

\noindent\begin{corollary}\label{co:3-AP} For every integer $r\geq 3$, we have $g_r([n])= \left(\binom{r}{2}+o(1)\right) 2^{n}.$
Moreover, all but a $o(1)$ proportion of rainbow 3-AP-free $r$-colorings of $[n]$ are 2-colorings.
\end{corollary}

Next, we consider the following natural question: among all subsets of $[n]$, which one admits the maximum number of rainbow 3-AP-free $r$-colorings? Using Szemer\'{e}di's theorem and Theorem \ref{th:3-AP}, we will answer this question by proving the following result in Section~\ref{sec:proof}.

\noindent\begin{theorem}\label{th:3-APsubset} For all integers $r\geq 3$, there exists $n_1\in\mathbb{N}$ such that for all $n\geq n_1$ and any proper subset $A\subset [n]$, we have $g_r(A)<g_r([n])$.
\end{theorem}

Thus $[n]$ is the unique subset admitting the maximum number of rainbow 3-AP-free $r$-colorings among all subsets of $[n]$.

Let $\mathbb{Z}_n$ be the cyclic group of order $n$ formed by the set $\{0, 1, \ldots, n-1\}$ with the binary operation addition modulo $n$. A 3-AP in $\mathbb{Z}_n$ is a sequence $a, b, c$ such that $a+c\equiv 2b \pmod{n}$. For any $A\subseteq \mathbb{Z}_n$ and $r\geq 3$, we use $g_r(A, \mathbb{Z}_n)$ to denote the number of rainbow 3-AP-free $r$-colorings of $A$, and we use the shorthand $g_r(\mathbb{Z}_n)$ for $g_r(\mathbb{Z}_n, \mathbb{Z}_n)$. For any $A\subseteq [n]$, if we also view $A$ as a subset of $\mathbb{Z}_n$, then obviously $g_r(A, \mathbb{Z}_n)\leq g_r(A, [n])$. Moreover, the lower bound $g_r(A, \mathbb{Z}_n)\geq \binom{r}{2}2^{|A|}-r^2+2r$ also holds. Hence, we have the following corollary.

\noindent\begin{corollary}\label{co:Zn} For every integer $r\geq 3$, we have
$$\binom{r}{2}2^{n}-r^2+2r\leq g_r(\mathbb{Z}_n)\leq \binom{r}{2} 2^{n}+2^{-\frac{n}{36\log_2n}}2^n,$$
and thus all but a $o(1)$ proportion of rainbow 3-AP-free $r$-colorings of $\mathbb{Z}_n$ are 2-colorings. Moreover, we have $g_r(A, \mathbb{Z}_n)< g_r(\mathbb{Z}_n)$ for any proper subset $A\subset \mathbb{Z}_n$.
\end{corollary}

Given two positive integers $n$ and $k$, the \emph{anti-van der Waerden number} $aw([n], k)$ (resp., $aw(\mathbb{Z}_n, k)$) is the smallest $r$ such that every exact $r$-coloring of $[n]$ (resp., $\mathbb{Z}_n$) contains a rainbow $k$-AP, where an \emph{exact $r$-coloring} is a coloring using all the $r$ colors. For an integer $n$, if $aw(\mathbb{Z}_n, 3)=3$, then all the rainbow 3-AP-free $r$-colorings of $\mathbb{Z}_n$ are 2-colorings, so $g_r(\mathbb{Z}_n)=\binom{r}{2}2^n-r^2+2r$. In \cite{JLMNR}, Jungi\'{c} et al. gave a characterization of integers $n$ such that $aw(\mathbb{Z}_n, 3)=3$. Moreover, Butler et al. \cite{BEHH} proved that $3\leq aw(\mathbb{Z}_p, 3)\leq 4$ for any prime $p$. For any positive integer $n$, let $\mathbb{Z}_n^\times$ be the multiplicative group of integers modulo $n$, and let ${\rm ord}_n(2)$ be the order of 2 in $\mathbb{Z}_n^\times$. Using their results, we can give the exact value of $g_r(\mathbb{Z}_p)$ for any odd prime.

\noindent\begin{theorem}\label{th:exactZp} For any prime $p\geq 3$ and integer $r\geq 3$, we have $$g_r(\mathbb{Z}_p)= \binom{r}{2}2^p-r^2+2r+ r\binom{r-1}{2}p\Big(2^{\frac{p-1}{c \cdot {\rm ord}_p(2)}}-2\Big),$$ where $c=1$ if ${\rm ord}_p(2)$ is even, and $c=2$ otherwise.
\end{theorem}

If $aw(\mathbb{Z}_p, 3)=3$, then $g_r(\mathbb{Z}_p)= \binom{r}{2}2^p-r^2+2r$, and Theorem~\ref{th:exactZp} is trivial in this case. If $aw(\mathbb{Z}_p, 3)=4$, then Theorem~\ref{th:exactZp} implies that there are exactly $r\binom{r-1}{2}p\Big(2^{\frac{p-1}{c \cdot {\rm ord}_p(2)}}-2\Big)$ exact 3-colorings of $\mathbb{Z}_p$ with no rainbow 3-AP.

The remainder of this paper is organized as follows. In the next section, we state and prove a container theorem for rainbow 3-AP-free colorings. In Section \ref{sec:proof}, we give our proofs of Theorems \ref{th:3-AP} and \ref{th:3-APsubset}. In Section \ref{sec:proofZp}, we prove Theorem \ref{th:exactZp}; there we also present some related results.

\section{Notation and preliminaries}
\label{sec:pre}

We will prove Theorem \ref{th:3-AP} using the Hypergraph Container Method developed by Balogh-Morris-Samotij \cite{BaMS} and Saxton-Thomason \cite{SaTh} independently. We first introduce some additional notation.

Let $\mathcal{H}$ be a $k$-uniform hypergraph. Let $d(\mathcal{H})$ be the \emph{average degree} of $\mathcal{H}$. For a subset $U\subseteq V(\mathcal{H})$, let $\mathcal{H}[U]$ be the subhypergraph of $\mathcal{H}$ induced by $U$, and let $d(U)=|\{e\in E(\mathcal{H})\colon\, U\subseteq e\}|$ be the \emph{co-degree} of $U$. For $2\leq j\leq k$, the \emph{$j$th maximum co-degree} of $\mathcal{H}$ is $\Delta_j(\mathcal{H})=\max\{d(U)\colon\, U\subseteq V(\mathcal{H}), |U|=j\}$. When the underlying hypergraph is clear, we simply write $d(\mathcal{H})$ and $\Delta_j(\mathcal{H})$ as $d$ and $\Delta_j$, respectively. For $0<\tau <1$, the \emph{co-degree function} is defined to be
$$\Delta(\mathcal{H},\tau)=2^{\binom{k}{2}-1}\sum^{k}_{j=2}2^{-\binom{j-1}{2}} \frac{\Delta_j}{d\tau^{j-1}}.$$

We will use the following form of the hypergraph container theorem.

\noindent\begin{theorem}\label{th:HCT} {\normalfont (Hypergraph container theorem, \cite[Corollary~3.6]{SaTh})}
Let $\mathcal{H}$ be a $k$-uniform hypergraph on vertex set $[N]$. Let $0< \varepsilon, \tau< 1/2$. Suppose that $\tau  < 1/(200k!^2k)$ and $\Delta(\mathcal{H}, \tau) \leq  \varepsilon /(12k!)$. Then there exists $c = c(k) \leq  1000k!^3k$ and a collection $\mathcal{C}$ of vertex subsets  such that
\begin{itemize}
  \item[{\rm (i)}] every independent set in $\mathcal{H}$ is a subset of some $C\in \mathcal{C}$;

  \item[{\rm (ii)}] for every $C \in \mathcal{C}$, $|E(\mathcal{H} [C])| \leq  \varepsilon \cdot |E(\mathcal{H})|$;

  \item[{\rm (iii)}] $\log_2 |\mathcal{C}| \leq cN\tau \log_2(1/\varepsilon) \log_2(1/\tau)$.
\end{itemize}
\end{theorem}

A key concept in applying the hypergraph container method to colored structures is the notion of a \emph{template}, which was first introduced in \cite{FaOU}.

\noindent\begin{definition}\label{de:template} {\normalfont (Template, palette, subtemplate, rainbow 3-AP)}

{\rm \begin{itemize}
       \item[(1)] An \emph{$r$-template} of order $n$ is a function $P\colon\, [n] \to  2^{[r]}$, associating to each integer $x$ in $[n]$ a list of colors $P(x) \subseteq  [r]$. We refer to this set $P(x)$ as the \emph{palette} available at $x$.

       \item[(2)] Let $P_1, P_2$ be two $r$-templates of order $n$. We say that $P_1$ is a {\it subtemplate} of $P_2$ (written as $P_1 \subseteq P_2$) if $P_1(x) \subseteq P_2(x)$ for every integer $x\in [n]$.

       \item[(3)] For an $r$-template $P$ of order $n$, we say that $P$ is a \emph{rainbow 3-AP} if there exist three integers $a, b, c$ in $[n]$ with $a+c=2b$ such that $|P(a)|=|P(b)|=|P(c)|=1$, $|P(x)|=0$ for each $x\in [n]\setminus \{a, b, c\}$, and $P(a)$, $P(b)$, $P(c)$ are pairwise distinct.

       \item[(4)] For an $r$-template $P$, we say that $P$ is \emph{rainbow 3-AP-free} if there is no subtemplate that is a rainbow 3-AP.
     \end{itemize}}
\end{definition}

Note that for any $A\subseteq [n]$, an $r$-coloring of $A$ can be viewed as an $r$-template $P$ with $|P(x)|=1$ for each $x\in A$ and $|P(x)|=0$ for each $x\in [n]\setminus A$. For an $r$-template $P$, let $R(P)$ be the number of subtemplates of $P$ that are rainbow 3-APs. We should remark that we only consider nontrivial 3-APs, that is, we do not consider the 3-APs with common difference 0. Let $f(n)$ be the number of 3-APs in $[n]$. Note that there is a bijection between all the 3-APs and ordered pairs $(a, b)\in [n]^2$ with $a<b$ and $a\equiv b \pmod{2}$. Thus $f(n)=\binom{\left\lceil n/2 \right\rceil}{2}+ \binom{\left\lfloor n/2 \right\rfloor}{2}= \left\lfloor\frac{n}{2}\right\rfloor\cdot \left\lfloor\frac{n-1}{2}\right\rfloor$. Throughout this paper, we will assume that $n$ is large enough. Using Theorem~\ref{th:HCT}, we can prove the following container theorem for rainbow 3-AP-free colorings.

\noindent\begin{theorem}\label{th:3-APCT}
For every integer $r\geq 3$, there exists $c = c(r)$ and a collection $\mathcal{C}$ of $r$-templates of order $n$ such that
\begin{itemize}
  \item[{\rm (i)}] every rainbow 3-AP-free $r$-template of order $n$ is a subtemplate of some $P\in \mathcal{C}$;

  \item[{\rm (ii)}] for every $P \in \mathcal{C}$, $R(P) \leq n^{-1/3}f(n)$;

  \item[{\rm (iii)}] $\log_2 |\mathcal{C}| \leq cn^{2/3}(\log_2 n)^2$.
\end{itemize}
\end{theorem}

\begin{proof} Let $\mathcal{H}$ be a 3-uniform hypergraph with vertex set $[n] \times [r]$, whose edges are all triples $\{(a, i), (b, j), (c, \ell)\}$ such that $a, b, c$ form a 3-AP in $[n]$ and $i, j, \ell$ are three pairwise distinct colors in $[r]$. In other words, every hyperedge in $\mathcal{H}$ corresponds to a rainbow 3-AP in $[n]$. Note that every vertex subset of $\mathcal{H}$ corresponds to an $r$-template of order $n$, and every independent set in $\mathcal{H}$ corresponds to a rainbow 3-AP-free $r$-template of order $n$. Hence, it suffices to show that for appropriate $\varepsilon$ and $\tau$, there exists a collection $\mathcal{C}$ of vertex subsets satisfying Theorem~\ref{th:HCT}~(i)-(iii). To achieve this, we need to check that $\Delta(\mathcal{H}, \tau) \leq  \varepsilon /(12\cdot 3!)$.

Since there are exactly $r(r-1)(r-2)$ ways to rainbow color a 3-AP with $r$ colors, the average degree $d$ of $\mathcal{H}$ satisfies
$$d=\frac{3|E(\mathcal{H})|}{|V(\mathcal{H})|}=\frac{3r(r-1)(r-2)f(n)}{nr}\geq \frac{3(r-1)(r-2)n}{5}.$$
Note that we have $\Delta_2=3(r-2)$ and $\Delta_3=1$. Let $\varepsilon=n^{-1/3}/(r(r-1)(r-2))$ and $\tau=(48\cdot 3!\cdot r)^{1/2}n^{-1/3}$. Then
$$\Delta(\mathcal{H}, \tau)= \frac{4\Delta_2}{d\tau}+\frac{2\Delta_3}{d\tau^2}\leq \frac{20}{(r-1)n\tau}+\frac{10}{3(r-1)(r-2)n\tau^2}\leq \frac{11}{3(r-1)(r-2)n\tau^2}\leq \frac{\varepsilon}{12\cdot 3!}.$$
The result follows.
\end{proof}

\noindent\begin{definition}\label{de:goodtemplate} {\normalfont (Good $r$-template)} {\rm For any $A\subseteq [n]$, an $r$-template $P$ of order $n$ is a \emph{good $r$-template of $A$} if it satisfies the following properties:

\begin{itemize}
  \item[(1)] $|P(x)|\geq 1$ for every $x\in A$;

  \item[(2)] $R(P)\leq n^{-1/3}f(n)$.
\end{itemize}}
\end{definition}

For a collection $\mathcal{P}$ of $r$-templates of order $n$ and $A\subseteq [n]$, we use $G(\mathcal{P}, A)$ to denote the set of rainbow 3-AP-free $r$-colorings of $A$ that are subtemplates of some $P\in \mathcal{P}$. Let $g(\mathcal{P}, A)=|G(\mathcal{P}, A)|$. If $\mathcal{P}=\{P\}$, then we simply write $G(\mathcal{P}, A)$ and $g(\mathcal{P}, A)$ as $G(P, A)$ and $g(P, A)$, respectively.

\section{Proofs of Theorems~\ref{th:3-AP} and \ref{th:3-APsubset}}
\label{sec:proof}

We first state and prove some propositions and lemmas.

\noindent\begin{proposition}\label{prop:pairs} Let $I\subseteq A\subseteq [n]$. There are at most $\frac{|I|^2}{4}+\frac{|I|(n-|A|)}{2}$ ordered pairs $(a, b)\in I^2$ with $a<b$ such that $\{a, b\}$ is not contained in any 3-AP of $A$.
\end{proposition}

\begin{proof} For any ordered pair $(a, b)\in I^2$ with $(a+b)/2\in A$, we have that $\{a, (a+b)/2, b\}$ forms a 3-AP of $A$. Thus if $\{a, b\}$ is not contained in any 3-AP of $A$, then either $a$ and $b$ have different parities, or $(a+b)/2\in [n]\setminus A$. Let $\alpha$ be the number of odd integers in $I$. Then the number of ordered pairs $(a, b)\in I^2$ with $a<b$ such that $a$ and $b$ have different parities is at most $\alpha (|I|-\alpha)\leq |I|^2/4$. For any $i\in [n]\setminus A$, the number of ordered pairs $(a, b)\in I^2$ with $a<b$ and $(a+b)/2=i$ is at most $\min \{|[i-1]\cap I|, |([n]\setminus [i])\cap I|\}$. Thus the number of ordered pairs $(a, b)\in I^2$ with $a<b$ and $(a+b)/2\in [n]\setminus A$ is at most $|I|(n-|A|)/2$. The result follows.
\end{proof}

\noindent\begin{proposition}\label{prop:bipairs} For sufficiently large $n$, let $I\subseteq A\subseteq [n]$ with $|I|\geq \frac{73}{74}n$, and let $I_1$, $I_2$ be a bipartition of $I$ with $|I_1|\leq |I_2|$. There are at least $\frac{|I_1||I_2|}{9}-3|I_1|(n-|A|)$ pairs $(a, b)$ with $a\in I_1$ and $b\in I_2$ such that $\{a, b\}$ is contained in some 3-AP of $A$.
\end{proposition}

\begin{proof} Since $|I_1|\leq |I_2|$, we have $|I_1|\leq |I|/2\leq n/2$ and $|I_2|\geq |I|/2\geq 73n/148$. We first consider the pairs $(a, b)$ with $a\in I_1$ and $b\in I_2$ such that $\{a, b\}$ is contained in some 3-AP of $[n]$. For any pair $(i, j)$ with $1\leq i< j\leq n$, if $\{i, j\}$ is not contained in any 3-AP of $[n]$, then $2i-j\leq 0$, $2j-i\geq n+1$, and $i, j$ have different parities. So $i\leq j/2\leq n/2$, $j\geq 2i$ and $j\geq (n+1+i)/2$.

For each $i\in [n]$, the number of integers $j\in [n]$ with $j>i$ such that $\{i, j\}$ is not contained in any 3-AP of $[n]$ is at most $\left\lceil\frac{1}{2}\left(n- \max\left\{\min \{2i-1, n\}, \frac{n+1+i}{2}-1\right\}\right)\right\rceil\leq \alpha(i)$, where $\alpha(i):=\frac{1}{2}\left(n- \max\left\{2i-1, \frac{n-1+i}{2}\right\}+1\right)$ if $1\leq i\leq \lfloor n/2\rfloor$, and $\alpha(i):=0$ if $\lfloor n/2\rfloor+1\leq i\leq n$. Note that $\alpha(1)\geq \alpha(2)\geq \cdots \geq \alpha(\lfloor n/2\rfloor)\geq \alpha(\lfloor n/2\rfloor+1)= \cdots = \alpha(n)=0$. Similarly, for each $j\in [n]$, the number of integers $i\in [n]$ with $i<j$ such that $\{i, j\}$ is not contained in any 3-AP of $[n]$ is at most $\beta (j):= \alpha (n+1-j)$. Note that $\beta(n)\geq \beta(n-1)\geq \cdots \geq \beta(\lceil n/2\rceil+1)\geq \beta(\lceil n/2\rceil)= \cdots = \beta(1)=0$.

From the above argument, we have that at least one of $\alpha(i)$ and $\beta(i)$ is zero for every $i\in [n]$. Moreover, the sequence of $\beta$s is exactly the same as the sequence of $\alpha$s but in reverse order. Furthermore, $\lceil |I_1|/2\rceil\leq \lceil |I|/4\rceil\leq \lceil n/4\rceil\leq (n+1)/3$. Thus the number of pairs $(a, b)$ with $a\in I_1$ and $b\in I_2$ such that $\{a, b\}$ is contained in some 3-AP of $[n]$ is at least
\begin{align*}
  ~ &~|I_1||I_2|-\sum_{a\in I_1} (\alpha(a)+\beta(a))= ~|I_1||I_2|-\sum_{a\in I_1} \max\{\alpha(a), \beta(a)\}\\
   \geq &~|I_1||I_2|-\big(\alpha(1)+\alpha(2)+\cdots +\alpha(\lceil |I_1|/2\rceil)+\beta(n)+\beta(n-1)+\cdots+\beta(n+1-\lfloor|I_1|/2\rfloor)\big)\\
   \geq &~|I_1||I_2|-2\sum^{\lceil |I_1|/2\rceil}_{i=1}\alpha(i)=~ |I_1||I_2|-2\sum^{\lceil |I_1|/2\rceil}_{i=1}\frac{1}{2}\left(n-\frac{n-1+i}{2}+1\right)\\
   \geq &~|I_1||I_2|-\frac{|I_1|+1}{2}\cdot \frac{1}{2}\cdot \left(\frac{n+2}{2}+\frac{n-|I_1|/2+3}{2}\right)=~|I_1||I_2|-\frac{(|I_1|+1)(2n-|I_1|/2+5)}{8}\\
   \geq &~\frac{|I_1|}{4}\left(4|I_2|-2n+\frac{|I_1|}{2}-5\right)=~ \frac{|I_1|}{4}\left(\frac{4}{9}|I_2|+\frac{37}{18}|I_2|+\frac{3}{2}|I_2|+\frac{|I_1|}{2}-2n-5\right)\\
   \geq &~\frac{|I_1|}{4}\left(\frac{4}{9}|I_2|+\frac{37}{36}|I|+|I|-2n-5\right) \geq~\frac{|I_1|}{4}\left(\frac{4}{9}|I_2|+\frac{73}{36}\cdot \frac{73}{74}n-2n-5\right)\geq~\frac{|I_1||I_2|}{9}.
\end{align*}

Next, we consider the pairs $(a, b)$ with $a\in I_1$ and $b\in I_2$ such that $\{a, b\}$ is contained in some 3-AP of $[n]$ but $\{a, b\}$ is not contained in any 3-AP of $A$, and we use $\gamma$ to denote the number of such pairs. In this case, at least one of the following holds: (i) $2a-b \in [n]\setminus A$, (ii) $(a+b)/2 \in [n]\setminus A$, (iii) $2b-a\in [n]\setminus A$. Thus $\gamma \leq 3|I_1|(n-|A|)$. The result follows.
\end{proof}

\noindent\begin{lemma}\label{lemma:x3} For sufficiently large $n$, let $r\geq 3$, $\delta=\frac{1}{34\log_2 n}$, $0\leq \xi \leq \frac{3}{13+8\log_2 r}$ and $A\subseteq [n]$ with $|A|= (1-\xi)n$. If there exists a good $r$-template $P$ of $A$ with $g(P, A)>2^{(1-\delta)n}$, then $\xi <(\log_2r-1)n^{-1/3}\log_2n+\delta$ and the number of integers $x\in A$ with $|P(x)|\geq 3$ is at most $n^{2/3}\log_2 n$.
\end{lemma}

\begin{proof} Let $X_1=\{x\in A\colon\, |P(x)|=1\}$, $X_2=\{x\in A\colon\, |P(x)|=2\}$ and $X_3=\{x\in A\colon\, |P(x)|\geq 3\}$. For $1\leq i\leq 3$, let $x_i=|X_i|$. Since $g(P, A)>2^{(1-\delta)n}$, we have $1^{x_1}\cdot 2^{x_2}\cdot r^{x_3}>2^{(1-\delta)n}$. Since $\sum_{1\leq i\leq 3}x_i = (1-\xi)n$, we have
\begin{equation}\label{eq:2}
x_3\log_2 r>(\xi-\delta)n+x_1+x_3.
\end{equation}

We claim that $x_2\geq n/4+3\xi n+n/\log_2 n+1$. Otherwise, we have $x_1+x_3>3n/4-4\xi n-n/\log_2 n-1$, so $x_3>(3/4-3\xi-o(1))n/\log_2 r=\Theta(n)$ by inequality (\ref{eq:2}). Note that if $(a, b)\in X_3\times X_3$ and $\{a, b, c\}$ forms a 3-AP for some $c\in A$, then $P$ contains a subtemplate $P'$ such that $P'$ is a rainbow 3-AP with $|P'(a)|=|P'(b)|=|P'(c)|=1$. By Proposition \ref{prop:pairs} and our choice of $\xi$, we have
\begin{align*}
  R(P)\geq &~\frac{1}{3}\left(\binom{x_3}{2}-\frac{(x_3)^2}{4}-\frac{x_3\xi n}{2}\right)=~\frac{x_3}{12}\left(x_3-2\xi n-2\right)\\
   > &~\frac{x_3}{12}\left(\frac{n}{\log_2 r}\left(\frac{3}{4}-3\xi-o(1)\right)-2\xi n-2\right)=~\frac{x_3 n}{12\log_2 r}\left(\frac{3}{4}-3\xi-2\xi\log_2 r-o(1)\right)\\
   \geq &~\frac{x_3 n}{12\log_2 r}\left(\frac{3}{4}-\frac{9}{13+8\log_2 r}-\frac{6\log_2 r}{13+8\log_2 r}-o(1)\right)=~\Theta(n^2),
\end{align*}
contradicting the assumption that $P$ is a good $r$-template.

For any $a\in X_3$, let $Y_a=\{b\in X_2\colon\, \{a, b\}\ \mbox{is contained in some 3-AP of } A\}$. For any $b\in X_2\setminus Y_a$, we have $2a-b\notin A$, $2b-a\notin A$ and $(a+b)/2\notin A$. Since $|\{b\in X_2\setminus Y_a\colon\, 2a-b\notin [n], 2b-a\notin [n], (a+b)/2\notin [n]\}|\leq \lceil n/4\rceil$ and $|\{b\in X_2\setminus Y_a\colon\, 2a-b\in [n]\setminus A\}|+|\{b\in X_2\setminus Y_a\colon\, 2b-a\in [n]\setminus A\}|+|\{b\in X_2\setminus Y_a\colon\, (a+b)/2\in [n]\setminus A\}|\leq 3\xi n$, we have $|X_2\setminus Y_a|\leq \lceil n/4\rceil+3\xi n\leq n/4+3\xi n+1$. Thus $|Y_a|\geq x_2-n/4-3\xi n-1\geq n/\log_2 n$. Note that if $a\in X_3$, $b\in X_2$ and $\{a, b, c\}$ forms a 3-AP for some $c\in A$, then $P$ contains a subtemplate $P'$ such that $P'$ is a rainbow 3-AP with $|P'(a)|=|P'(b)|=|P'(c)|=1$. Thus $R(P)\geq x_3n/(2\log_2 n)$. Since $P$ is a good $r$-template, we have $n^{-1/3}f(n)\geq x_3n/(2\log_2 n)$. Thus $x_3\leq n^{2/3}\log_2 n$. Moreover, by inequality (\ref{eq:2}), we have $\xi< (\log_2r-1)n^{-1/3}\log_2 n+\delta$.
\end{proof}

\noindent\begin{lemma}\label{lemma:stability} {\normalfont (Stability)} For sufficiently large $n$, let $r\geq 3$, $\delta=\frac{1}{34\log_2 n}$, $0\leq \xi < (\log_2r-1)n^{-1/3}\log_2n+\delta$ and $A\subseteq [n]$ with $|A|= (1-\xi)n$. Suppose $P$ is a good $r$-template of $A$ with $g(P, A)>2^{(1-\delta)n}$. Then there exist two colors $i, j\in [r]$ such that the number of integers $x\in A$ with $P(x)=\{i,j\}$ is at least $(1-2\delta)n$.
\end{lemma}

\begin{proof} Let $X_i$ and $x_i$ ($1\leq i\leq 3$) be the sets and set sizes as given in the proof of Lemma \ref{lemma:x3}. By Lemma \ref{lemma:x3} and inequality (\ref{eq:2}), we have $x_3\leq n^{2/3}\log_2 n$ and $x_1+x_3<x_3\log_2 r-(\xi-\delta)n$. Thus $x_2> (1-\delta)n-x_3\log_2r\geq (1-\delta- n^{-1/3}\log_2 n\log_2 r)n$. For $1\leq i< j\leq r$, let $Y_{i,j}=\{a\in X_2\colon\, P(a)=\{i,j\}\}$. Without loss of generality, let $|Y_{1,2}|\geq x_2/\binom{r}{2}$. Let $Y'=X_2\setminus Y_{1,2}$. Note that if $a\in Y_{1,2}$, $b\in Y'$ and $\{a, b, c\}$ forms a 3-AP for some $c\in A$, then $P$ contains a subtemplate $P'$ such that $P'$ is a rainbow 3-AP with $|P'(a)|=|P'(b)|=|P'(c)|=1$. By Proposition \ref{prop:bipairs} and since $P$ is a good $r$-template, we have $n^{-1/3}n^2/4 \geq R(P)\geq (|Y_{1,2}||Y'|/9-3\xi n\cdot \min\{|Y_{1,2}|, |Y'|\})/2$.

If $|Y_{1,2}|\leq |Y'|$, then
$$|Y'|\leq \frac{9n^{5/3}}{2|Y_{1, 2}|}+27\xi n\le \frac{9n^{5/3}\binom{r}{2}}{2x_2}+27\xi n < \frac{9n^{2/3}\binom{r}{2}}{2(1-\delta- n^{-1/3}\log_2 n\log_2 r)}+27\xi n= o(n).$$
Note that in this case we have $|Y'|\geq x_2/2> (1-\delta- n^{-1/3}\log_2 n\log_2 r)n/2=\Theta(n)$, a contradiction. If $|Y_{1, 2}|\geq |Y'|$, then
$$|Y'|\leq \frac{n^{5/3}}{2(|Y_{1,2}|/9-3\xi n)}\leq \frac{n^{5/3}}{x_2/9-6\xi n} < \frac{n^{2/3}}{(1-\delta- n^{-1/3}\log_2 n\log_2 r)/9-6\xi}= O(n^{2/3}).$$
Thus $|Y_{1,2}|\geq x_2-|Y'|\geq (1-2\delta)n$.
\end{proof}

\noindent\begin{lemma}\label{lemma:counting} For sufficiently large $n$, let $r\geq 3$, $\delta=\frac{1}{34\log_2 n}$, $0\leq \xi < (\log_2r-1)n^{-1/3}\log_2n+\delta$ and $A\subseteq [n]$ with $|A|= (1-\xi)n$. For any two colors $i, j\in [r]$, let $\mathcal{P}_{i,j}$ be the set of good $r$-templates of $A$, in which there are at least $(1-2\delta)n$ integers with palette $\{i, j\}$. Then $g(\mathcal{P}_{i,j}, A)\leq 2^{|A|}(1+2^{-n/240})$.
\end{lemma}

\begin{proof} For any rainbow 3-AP-free $r$-coloring $\sigma\in G(\mathcal{P}_{i,j}, A)$, let $I_{\sigma}=\{x\in A \colon\, \sigma(x)\notin \{i, j\}\}$. Then $|I_{\sigma}|\leq 2\delta n$. Let $\mathcal{G}_0=\{\sigma \in G(\mathcal{P}_{i,j}, A)\colon\, |I_{\sigma}|=0\}$ and $\mathcal{G}_1=\{\sigma \in G(\mathcal{P}_{i,j}, A)\colon\, |I_{\sigma}|\geq 1\}$. Then $g(\mathcal{P}_{i,j}, A)=|\mathcal{G}_0|+|\mathcal{G}_1|\leq 2^{|A|}+|\mathcal{G}_1|$.

Next, we consider the number of colorings in $\mathcal{G}_1$. We first choose a subset $I_0\subseteq A$ with $1\leq |I_0|\leq 2\delta n$. The number of options is at most $\sum_{1\leq \ell\leq 2\delta n}\binom{n}{\ell}$. We then use colors in $[r]\setminus \{i,j\}$ to color a fixed $I_0$. The number of colorings is at most $(r-2)^{2\delta n}$. Now we consider the number of ways to color $A\setminus I_0$ using color $i$ or $j$. For any fixed $x\in I_0$, let $d(x)$ be the number of ordered pairs $(a, b)\in (A\setminus I_0)^2$ with $a<b$ such that $\{a, b, x\}$ forms a 3-AP. We claim that $d(x)\geq n/5$. In fact, if we let $d_1=\left|\left\{(a, b)\in A^2\colon\, x<a<b,\ x+b=2a\right\}\right|$, $d_2=\left|\left\{(a, b)\in A^2\colon\, a<x<b,\ a+b=2x\right\}\right|$ and $d_3=\left|\left\{(a, b)\in A^2\colon\, a<b<x,\ a+x=2b\right\}\right|$, then $d(x)\geq d_i-2|I_0|\geq d_i-4\delta n$ for $1\leq i\leq 3$. Note that $d_1\geq\lfloor(n+x)/2\rfloor-x-\xi n$, $d_2\geq\min \{x-1, n-x\}-\xi n$ and $d_3\geq x-\lceil(1+x)/2\rceil-\xi n$. If $1\leq x\leq 2n/5$, then $d(x)\geq d_1-4\delta n\geq (n-x-1)/2-\xi n-4\delta n\geq n/5$. If $2n/5< x\leq 3n/5$, then $d(x)\geq d_2-4\delta n\geq \lfloor2n/5\rfloor-\xi n-4\delta n\geq n/5$. If $3n/5< x\leq n$, then $d(x)\geq d_3-4\delta n\geq (x-2)/2-\xi n - 4\delta n\geq n/5$. Hence, $d(x)\geq n/5$. Note that if $(a, b)\in (A\setminus I_0)^2$ and $\{a, b, x\}$ forms a 3-AP, then $\sigma(a)\in \{i,j\}$, $\sigma(b)\in \{i,j\}$ and $\sigma(x)\in [r]\setminus \{i,j\}$, so $\sigma(a)=\sigma(b)$ in order to avoid a rainbow 3-AP. Thus the number of ways to color $A\setminus I_0$ is at most $2^{|A|-|I_0|-d(x)/3}\leq 2^{|A|-n/15}$.

Thus
\begin{align*}
  g(\mathcal{P}_{i,j}, A)\leq &~2^{|A|}+|\mathcal{G}_1|\leq~ 2^{|A|}+\sum_{1\leq \ell\leq 2\delta n}\binom{n}{\ell}(r-2)^{2\delta n}2^{|A|-n/15}\\
   \leq &~2^{|A|}+2\delta n \frac{n^{2\delta n}}{(2\delta n)!}r^{2\delta n}2^{|A|-n/15}\leq~ 2^{|A|}+2^{2\delta n\log_2n}2^{2\delta n\log_2r}2^{|A|-n/15}\\
   \leq &~2^{|A|}+2^{n/17+(n\log_2 r)/(17\log_2 n)+|A|-n/15}\leq~ 2^{|A|}+2^{|A|-n/240}.
\end{align*}
This completes the proof of Lemma~\ref{lemma:counting}.
\end{proof}

Now we have all the ingredients to present our proof of Theorem \ref{th:3-AP}.

\begin{proof}[Proof of Theorem~\ref{th:3-AP}] Let $\mathcal{C}$ be the collection of $r$-templates given by Theorem \ref{th:3-APCT}. Let $\delta= (34\log_2n)^{-1}$, and let $\mathcal{C}_1=\{P\in\mathcal{C}\colon\, g(P, A)\leq 2^{(1-\delta)n}\}$ and $\mathcal{C}_2=\{P\in\mathcal{C}\colon\, g(P, A)> 2^{(1-\delta)n}\}$. For each $P\in \mathcal{C}_2$, we have $|P(x)|\geq 1$ for every $x\in A$; otherwise $g(P, A)=0$. Moreover, we have $R(P)\leq n^{-1/3}f(n)$ by Theorem \ref{th:3-APCT} (ii). Thus every element $P\in \mathcal{C}_2$ is a good $r$-template of $A$. By Lemmas \ref{lemma:x3}, \ref{lemma:stability} and \ref{lemma:counting}, we have
$$g(\mathcal{C}_2, A)\leq \sum_{1\leq i<j\leq r}2^{|A|}\left(1+2^{-n/240}\right)\leq \binom{r}{2} 2^{|A|}\left(1+2^{-n/240}\right).$$
By Theorem \ref{th:3-APCT} (iii), we have
$$g(\mathcal{C}_1, A)\leq |\mathcal{C}_1|2^{(1-\delta)n}\leq |\mathcal{C}|2^{(1-\delta)n}\leq 2^{cn^{2/3}(\log_2 n)^2+(1-1/(34\log_2n))n}<2^{(1-1/(35\log_2n))n}.$$
Combining the above and using Theorem \ref{th:3-APCT} (i), the number of rainbow 3-AP-free $r$-colorings of $A$ is at most
$$g(\mathcal{C}_1, A)+g(\mathcal{C}_2, A)\leq 2^{(1-1/(35\log_2n))n}+\binom{r}{2} 2^{|A|}\left(1+2^{-n/240}\right)\leq \binom{r}{2} 2^{|A|}+2^{-\frac{n}{36\log_2n}}2^n.$$
This completes the proof of Theorem~\ref{th:3-AP}.
\end{proof}

Next, we present our proof of Theorem \ref{th:3-APsubset}.

\begin{proof}[Proof of Theorem~\ref{th:3-APsubset}] Let $n$ be a sufficiently large integer. For any subset $A\subset [n]$ with $(10+8\log_2 r)n/(13+8\log_2 r)\leq |A|< n$, by Theorem \ref{th:3-AP}, we have $g_r(A)\leq \binom{r}{2} 2^{n-1}+2^{-n/(36\log_2n)}2^n\leq (r(r-1)/4+o(1)) 2^{n}< \binom{r}{2}2^{n}-r^2+2r\leq g_r([n])$. For any subset $A\subset [n]$ with $|A|\leq n/\log_2 r$, we have $g_r(A)\leq r^{|A|}\leq 2^{|A|\log_2 r}\leq 2^n<g_r([n])$. In the following, we may assume that $|A|=\alpha n$ with $1/\log_2 r < \alpha < (10+8\log_2 r)/(13+8\log_2 r)$.

Let $k=\max\{n_0, 420(\log_2 r)^2\}$, where $n_0$ is the obtained value when applying Theorem~\ref{th:3-AP} with $\xi=3/(13+8\log_2 r)$. Let $\delta=1/(15(\log_2 r)^2)$ and $n\geq n_1=sz(k, \delta)$. By Szemer\'{e}di's theorem, $A$ contains at least $(\alpha -\delta)n/k$ pairwise disjoint $k$-APs. For any $k$-AP $F=\{x, x+y, \ldots, x+(k-1)y\}$, let $\varphi$ be a mapping between $[k]$ and $F$ such that $\varphi (i)=x+(i-1)y$ for every $i\in [k]$. Then there is a bijection $\phi$ between the 3-APs in $[k]$ and the 3-APs in $F$: $\phi (\{a, b, c\})=\{\varphi(a), \varphi(b), \varphi(c)\}$ for any 3-AP $\{a, b, c\}$ in $[k]$. Therefore, there is a bijection between the rainbow 3-AP-free $r$-colorings of $[k]$ and that of $F$, and thus $g_r(F)=g_r([k])$. By the choice of $k$, we have $g_r(F)=g_r([k])\leq \binom{r}{2} 2^{k}+2^{-k/(36\log_2k)}2^k\leq 2^{k+2\log_2 r}$. Then
\begin{align*}
  g_r(A)\leq &~r^{\delta n}\left(2^{k+2\log_2 r}\right)^{\frac{(\alpha -\delta)n}{k}}=~2^{\delta n \log_2 r}2^{(\alpha -\delta)n+\frac{2(\alpha -\delta)n\log_2 r}{k}}\\
   < &~2^{\alpha n+\delta n\log_2 r+\frac{2\alpha n \log_2 r}{k}}<~2^{\frac{(10+8\log_2 r)n}{13+8\log_2 r}+\frac{n}{15\log_2 r}+\frac{2\alpha n \log_2 r}{k}}\\
   < &~2^{\frac{(10+8\log_2 r)n}{13+8\log_2 r}+\frac{n}{14\log_2 r}}<~2^n<~g_r([n]).
\end{align*}
This completes the proof of Theorem~\ref{th:3-APsubset}.
\end{proof}

\section{Colorings of $\mathbb{Z}_n$ with no rainbow 3-AP}
\label{sec:proofZp}

We begin with the following two lemmas.

\noindent\begin{lemma}\label{le:Zn3} {\normalfont (\cite[Theorem~3.5]{JLMNR})} Let $n$ be a positive integer. Then $aw(\mathbb{Z}_n, 3)=3$ if and only if one of the following holds:
\begin{itemize}
\item[{\rm (i)}] $n$ is a power of 2;
\item[{\rm (ii)}] $n$ is prime and 2 is a generator of $\mathbb{Z}^{\times}_n$;
\item[{\rm(iii)}] $n$ is prime, ${\rm ord}_n(2)=\frac{n-1}{2}$, and $\frac{n-1}{2}$ is odd.
\end{itemize}
\end{lemma}

\noindent\begin{lemma}\label{le:Zp3} {\normalfont (\cite[Proposition~3.5]{BEHH})} For any prime $p$, we have $3\leq aw(\mathbb{Z}_p, 3)\leq 4$, and $aw(\mathbb{Z}_p, 3)=4$ implies that every rainbow 3-AP-free coloring of $\mathbb{Z}_p$ using exactly three colors contains a color which is used exactly once.
\end{lemma}

Now we have all the ingredients to present our proof of Theorem \ref{th:exactZp}.

\begin{proof}[Proof of Theorem~\ref{th:exactZp}] If $aw(\mathbb{Z}_p, 3)=3$, then it suffices to show that $c \cdot {\rm ord}_p(2)=p-1$. By Lemma \ref{le:Zn3}, either ${\rm ord}_p(2)=p-1$, or ${\rm ord}_p(2)=(p-1)/2$ and $(p-1)/2$ is odd. In both cases, we have $c \cdot {\rm ord}_p(2)=p-1$.

If $aw(\mathbb{Z}_p, 3)=4$, then by Lemma \ref{le:Zp3}, every rainbow 3-AP-free coloring of $\mathbb{Z}_p$ using exactly three colors contains a color which is used exactly once. Let $c$ be a rainbow 3-AP-free coloring of $\mathbb{Z}_p$ using exactly three colors. Without loss of generality, we may assume that $c(0)=1$ and $c(i)\in \{2, 3\}$ for all $i\in \mathbb{Z}_p\setminus \{0\}$. In order to avoid a rainbow 3-AP, we have $c(i)=c(-i)$ and $c(i)=c(2i)$ for each $i\in \mathbb{Z}_p\setminus \{0\}$.

For every integer $\ell$ with $1\leq \ell\leq \frac{p-1}{2}$, let $Q_{\ell}=\{(\ell\cdot 2^i)\bmod{p}\colon\, i\in \mathbb{Z}^+\}\cup \{-(\ell\cdot 2^i)\bmod{p}\colon\, i\in \mathbb{Z}^+\}$. For every integer $j$ with $1\leq j\leq \frac{p-1}{2}$, let $I_1=\{1\}$, $I_j=\{j\}$ if $j\geq 2$ and $j\notin \bigcup_{\ell \in \bigcup_{m=1}^{j-1}I_m}Q_{\ell}$, and $I_j=\emptyset$ otherwise. Then for every $1\leq \ell\leq \frac{p-1}{2}$, all the elements in $Q_{\ell}$ should be colored with the same color. For any $j$ with $2\leq j\leq \frac{p-1}{2}$, if $I_j=\emptyset$, then $j\in Q_{\ell}$ for some $\ell$ with $\ell \leq j-1$ and $I_{\ell}\neq \emptyset$, which implies that $Q_j\subseteq Q_{\ell}$. Thus we only need to consider the sets $Q_j$ and $I_j$ with $I_j\neq \emptyset$. Moreover, for any $j_1\neq j_2$ with $I_{j_1}\neq \emptyset$ and $I_{j_2}\neq \emptyset$, we have $Q_{j_1}\cap Q_{j_2}=\emptyset$. Note that $\bigcup_{\ell \in \bigcup_{j=1}^{(p-1)/2}I_j}Q_{\ell}= \mathbb{Z}_p \setminus \{0\}$. Furthermore, for any $\ell$ with $I_{\ell}\neq \emptyset$, we have $|Q_{\ell}|={\rm ord}_p(2)$ if ${\rm ord}_p(2)$ is even, and $|Q_{\ell}|=2 \cdot {\rm ord}_p(2)$ if ${\rm ord}_p(2)$ is odd.

Next, we show that if $\{0, a, b\}$ forms a 3-AP in $\mathbb{Z}_p$, then $a, b$ is contained in the same $Q_j$ for some $j$ with $1\leq j\leq \frac{p-1}{2}$ and $I_j\neq \emptyset$. For a contradiction, suppose that $a\in Q_{j_1}$ and $b\in Q_{j_2}$ with $Q_{j_1}\cap Q_{j_2}=\emptyset$. Then $2a \pmod{p}\in Q_{j_1}$, $2b \pmod{p}\in Q_{j_2}$ and $-b \pmod{p}\in Q_{j_2}$. Since $\{0, a, b\}$ forms a 3-AP, we have $2a\equiv b \pmod{p}$, $a\equiv 2b \pmod{p}$ or $a\equiv -b \pmod{p}$, which is impossible.

Now we have all the ingredients to count the number of rainbow 3-AP-free colorings of $\mathbb{Z}_p$ using exactly three colors. We first choose an element $x\in\mathbb{Z}_p$. The number of options is $p$. We next choose a color $i\in[r]$ to color the element $x$, and then choose a set of two distinct colors $\{j, k\}\subseteq [r]\setminus \{i\}$. The number of options is $r\binom{r-1}{2}$. Finally, for any fixed choices of $i$, $j$, $k$ and $x$ (say $i=1$, $j=2$, $k=3$ and $x=0$), we color $\mathbb{Z}_p\setminus \{x\}$ using colors $j$ and $k$, so that all the elements in $Q_{j}$ are colored with the same color for each $j$ with $I_{j}\neq \emptyset$, and each of the two colors is used at least once. The number of ways is $2^{\sum_{j=1}^{(p-1)/2}|I_j|}-2$. Therefore, we have
\begin{align*}
  g_r(\mathbb{Z}_p)= &~\binom{r}{2}2^p-r^2+2r+ r\binom{r-1}{2}p\left(2^{\sum_{j=1}^{(p-1)/2}|I_j|}-2\right)\\
   = &~\binom{r}{2}2^p-r^2+2r+ r\binom{r-1}{2}p\left(2^{\frac{p-1}{c \cdot {\rm ord}_p(2)}}-2\right).
\end{align*}
This completes the proof of Theorem~\ref{th:exactZp}.
\end{proof}

By Lemma \ref{le:Zn3} (i), if $n$ is a power of 2, then $g_r(\mathbb{Z}_n)=\binom{r}{2}2^{n}-r^2+2r$. Note that every integer $n$ ($\geq 3$) can be decomposed into prime factors. Using Theorem~\ref{th:exactZp}, we can derive the following recurrence inequality.

\noindent\begin{corollary}\label{co:Znp} For any positive integers $n\geq 1$, $r\geq 3$ and prime $p\geq 3$, we have
\begin{align*}
  g_r(\mathbb{Z}_{np})\geq &~\binom{r}{2}2^{np}-r^2+2r+pr(g_r(\mathbb{Z}_n)-(r-1)2^n+r-2)\\
    &~+\binom{r}{2}p\Big(2^{\frac{p-1}{c \cdot {\rm ord}_p(2)}}-2\Big)(g_r(\mathbb{Z}_n)-2^n),
\end{align*}
where $c=1$ if ${\rm ord}_p(2)$ is even, and $c=2$ otherwise.
\end{corollary}

\begin{proof} Note that there are exactly $\binom{r}{2}2^{np}-r^2+2r$ 2-colorings of $\mathbb{Z}_{np}$. In the following, we will construct colorings using at least three colors. Let $R_0, R_1, \ldots, R_{p-1}$ be the residue classes modulo $p$ in $\mathbb{Z}_{np}$, that is, $R_i=\{x\in \mathbb{Z}_{np}\colon\, x\equiv i \pmod{p}\}$ for each $i\in \{0, 1, \ldots, p-1\}$.

We choose an arbitrary integer $q\in \{0, 1, \ldots, p-1\}$ and an arbitrary color $\ell\in [r]$. Let $c_n$ be a rainbow 3-AP-free $r$-coloring of $\mathbb{Z}_n$ which contains at least two colors in $[r]\setminus \{\ell\}$. Let $c_{np}$ be a coloring of $\mathbb{Z}_{np}$ such that $c_{np}(ip+q)=c_n(i)$ for $i\in \{0, 1, \ldots, n-1\}$, and $c_{np}(u)=\ell$ for any $u\in \mathbb{Z}_{np}\setminus R_q$. Note that if $\{x, y, z\}$ forms a 3-AP in $\mathbb{Z}_{np}$ with $x, y\in R_q$, then $z\in R_q$ since $p$ is a prime. Thus $c_{np}$ is a rainbow 3-AP-free $r$-coloring of $\mathbb{Z}_{np}$. The number of such colorings $c_{np}$ is $pr(g_r(\mathbb{Z}_n)-r-(r-1)(2^n-2))= pr(g_r(\mathbb{Z}_n)-(r-1)2^n+r-2)$.

If $aw(\mathbb{Z}_p, 3)=4$, then we can further construct another family of colorings. By Lemma \ref{le:Zp3}, every rainbow 3-AP-free coloring of $\mathbb{Z}_p$ using exactly three colors contains a color which is used exactly once. We call the color used exactly once the special color, and the other two colors non-special. Let $c_p$ be a rainbow 3-AP-free coloring of $\mathbb{Z}_p$ using exactly three colors, and assume that the special color is used on $q$. Let $c_n$ be a rainbow 3-AP-free $r$-coloring of $\mathbb{Z}_n$ such that $c_n$ contains at least one color which is not one of the two non-special colors in $c_p$. Let $c_{np}$ be a coloring of $\mathbb{Z}_{np}$ such that $c_{np}(ip+q)=c_n(i)$ for $i\in \{0, 1, \ldots, n-1\}$, and $c_{np}(u)=c_p(j)$ for $u\in R_{j}$ and $j\in \{0, 1, \ldots, p-1\}\setminus \{q\}$. Note that if $\{x, y, z\}$ forms a 3-AP in $\mathbb{Z}_{np}$ with $x, y\in R_i$ for some $i\in \{0, 1, \ldots, p-1\}$, then $z\in R_i$ since $p$ is a prime. Thus there is no 3-AP in $\mathbb{Z}_{np}$ such that the three elements are contained in exactly two different residue classes.

If $\{x, y, z\}$ forms a 3-AP in $\mathbb{Z}_{np}$ such that $x, y, z$ are contained in the same residue class, then this 3-AP is not rainbow since $c_n$ is rainbow 3-AP-free. If $\{x, y, z\}$ forms a 3-AP in $\mathbb{Z}_{np}$ such that $x, y, z$ are contained in three different residue classes, then we may assume that $x\in R_{j_1}$, $y\in R_{j_2}$, $z\in R_{j_3}$ and $x+z\equiv 2y \pmod{np}$, where $j_1, j_2$ and $j_3$ are three pairwise distinct elements in $\{0, 1, \ldots, p-1\}$. Thus $x=i_1p+j_1$, $y=i_2p+j_2$, $z=i_3p+j_3$ for some integers $i_1, i_2$ and $i_3$. Since $x+z\equiv 2y \pmod{np}$, we may further assume that $i_1p+j_1+i_3p+j_3= 2(i_2p+j_2)+knp$ for some integer $k$. Thus $j_1+j_3-2j_2=(2i_2-i_1-i_3+kn)p$, which implies that $j_1+j_3\equiv 2j_2 \pmod{p}$, and thus $\{j_1, j_2, j_3\}$ forms a 3-AP in $\mathbb{Z}_p$. Since $c_p$ is rainbow 3-AP-free, the three colors $c_p(j_1)$, $c_p(j_2)$ and $c_p(j_3)$ are not pairwise distinct. Since $c_p(q)$ is a special color in $c_p$, we further have that $c_{np}(x)$, $c_{np}(y)$ and $c_{np}(z)$ are not pairwise distinct. Then the 3-AP $\{x, y, z\}$ is not rainbow in $\mathbb{Z}_{np}$. Thus $c_{np}$ is a rainbow 3-AP-free $r$-coloring of $\mathbb{Z}_{np}$.

By Theorem \ref{th:exactZp}, the number of colorings $c_{p}$ is $r\binom{r-1}{2}p\Big(2^{\frac{p-1}{c \cdot {\rm ord}_p(2)}}-2\Big)$. For any $q\in \{0, 1, \ldots, p-1\}$ and any two distinct colors $i, j\in [r]$, there are $r-2$ colorings $c_{p}$ such that the special color is used on $q$, colors $i$ and $j$ are two non-special colors, and the colorings of $\mathbb{Z}_p\setminus \{q\}$ are the same. Moreover, for any fixed $c_p$, the number of options of $c_n$ is $g_r(\mathbb{Z}_n)-2^n$. Therefore, the number of colorings $c_{np}$ is $\frac{1}{r-2}r\binom{r-1}{2}p\Big(2^{\frac{p-1}{c \cdot {\rm ord}_p(2)}}-2\Big)(g_r(\mathbb{Z}_n)-2^n)=\binom{r}{2}p\Big(2^{\frac{p-1}{c \cdot {\rm ord}_p(2)}}-2\Big)(g_r(\mathbb{Z}_n)-2^n)$. This completes the proof of Corollary~\ref{co:Znp}.
\end{proof}

\end{document}